\documentclass{amsart}
\usepackage{amsmath}
\usepackage{amssymb}
\usepackage{graphicx}
\usepackage{amsthm}
\usepackage{verbatim}
\usepackage{color}
\usepackage{enumerate}
\usepackage{amsfonts}
\usepackage[all]{xy}
\usepackage{fancyhdr}
\theoremstyle{plain}

\newtheorem{theorem}{Theorem}[subsection]

\theoremstyle{remark}
\newtheorem{remark}[theorem]{Remark}

\theoremstyle{definition}

\begin{document}
\title{A very simple estimate of the rational homological dimension of moduli spaces of Riemann surfaces with boundary and marked points}

\author{Hao Yu}

\address{Department of Mathematics, University of Minnesota, Minneapolis, MN, 55455, USA}
\email{dustless2014@163.com}




\begin{abstract}

The moduli spaces of compact and connected Riemann surfaces has been a central topic in modern mathematics. Hence their homological dimensions become important invariants. Motivated by the emergence of open-closed string theory and its mathematical counterparts, we give a coarse estimate of rational homological dimension of Riemann surfaces with possible boundary and marked points(can lie on both interior and boundary). We hope it will have applications in open-closed theory, for example, open-closed Gromov-Witten theory in the future.

\end{abstract}

\maketitle

\section{Introduction}
The moduli spaces of compact and connected Riemann surfaces with marked points has been a central and active topic in topology, differential geometry, algebraic geometry and mathematical physics. The estimate of its homological dimension, i.e., the greatest degree for which the homology group over a coefficient ring $R$, typically $Z$ or $Q$, doesn't vanish, is thus a fundamental interest. The estimate of relevant quantities along this line has been worked out before (for example see \cite{har86,m17}). There are several applications of this estimate in the literature. For example,
in \cite{cos07} Costello gives a uniqueness of the solutions of Quantum master equation in topological conformal field theory by using a coarse estimate of this homological dimension. In this paper, we give an analogue estimate of the rational homological dimensions of the moduli spaces of Riemann surfaces with (ordered or unordered) boundaries and with (ordered or unordered) punctures/marked points on the interior and boundaries of the surfaces. The latter is known as open-closed cased in mathematical physics or algebraic geometry literature. We refer the readers \cite{liu02} for the thorough study of these spaces. Although our estimate is still coarse, analogue to the works mentioned above, we hope it will be useful in open-closed theory, e.g., open-closed Gromove-Witten theory, or open-closed string theory. 

In the following sections, we actually give two estimates of the homological dimensions of moduli spaces of Riemann surfaces with boundary and punctures/marked points respectively. It turns out that the latter is stronger than
the former. However, the method used in the deduction of the first estimate is fundamental although still subtle, involving tools in both geometry(Teichmuller spaces and classifying spaces) and algebra(mapping class groups and spectral sequences), as well as the dedicate analysis of special cases. It has independent interests, and we believe it will be useful on its own. The method for the second estimate uses the established deep theory of virtual duality groups, so it will be a little simpler than the first.
\newline
\newline
\quad\quad From hyperbolic geometry and the theory of Riemann surfaces, we know that a Riemann
surface with $r$ boundary components and $n$ punctures is equivalent to the surface
equipped with a complete, finite area hyperbolic metric with geodesic boundary (to
see this, double the surface, and we have a unique hyperbolic geometric structure corresponding to the complex structure of the double, and the boundary is invariant under
the involution, so it is geodesic). Let $\mathcal{H}_S$ be the space of such metrics on the surface
$S$, let $Diff(S)$ denote the group of oriented preserving diffeomorphisms of $S$ which fixes
each puncture and boundary component, and $Diff^1_S$ be the subgroup consisting
of those diffeomorphisms which are isotropic via such diffeomorphisms to the identity. And
$Mod(S) = Diff(S)/Diff^1_S$ is known as Mapping class group of $S$. $Diff(S)$ and
$Diff^1_S$ acts on $\mathcal{H}_S$ (via metric pull back) and the quotient space $\mathcal{T}_S = \mathcal{H}_S/Diff^1_S$
is called Teichmuller space of $S$. The space $\mathcal{M}_S = \mathcal{H}_S/Diff(S)$ is the Moduli space of
$S$, whose points parameterize the isomorphism classes of complete, finite area hyperbolic
metrics on $S$ with ordered punctures and boundary components. From the definition,
we also know $\mathcal{T}_S/Mod(S) = \mathcal{M}_S$.

In the remaining of this paper, we will denote $S^{b,\vec{m}}_{g,n}$ to be a surface of genus $g$ with $b$ boundary components, labelled $1, 2,\dots ,b$ and $n$ interior punctures/marked points, $\vec{m} = (m_1,m_2,... ,m_b)$
punctures/marked points on the boundary, with $m_i$ punctures/marked points on the $i$-th
boundary component. And let $m = \sum^b_{i=1} m_i$. If $\vec{m} = \vec{0}$, we denote it as $S^b_{g,n}$, and furthermore if $b = 0$ or $n = 0$ or
both are 0 we denote it by $S_{g,n}$, $S^b_g$ or $S_g$, respectively.

\section{Estimate I}
We will prove in this section the following theorem.
\begin{theorem}\label{hd_1}
Let $M_{g,n}^{b,\vec{m}}$ be moduli space of Riemann surfaces of genus $g$ with $b$
boundary components, labelled $1, 2,\dots ,b$ and $n$ interior punctures which is also labelled,
$\vec{m} = (m_1,m_2,\cdots, m_b)$ punctures on the boundary, with $m_i$ punctures on the $i$-th boundary
component, $m = \sum^b_{i=1} m_i$, then we have the following estimate for the homology groups:
\begin{equation}
H_i(\mathcal{M}_{S^{b,\vec{m}}_{g,n}},Q) = 0 \quad \textit{for} \quad i \ge 6g - 7 + 2n + 3b + m
\end{equation}
except $(g,n,b,\vec{m}) = (0, 3, 0, 0),(0, 2, 1, 0),(0, 2, 1, 1),(0, 0, 2,((1, 0)),(0, 0, 2,(1, 1)),\\
(0, 1, 1, 1),(0, 1, 1, 2),(0, 0, 1, 3),(0, 0, 1, 4),(0, 4, 0, 0).$
\end{theorem}
For oriented surfaces of genus $g$ (with or without boundaries and punctures), it is
well known that the corresponding Teichmuller space is contractible (it is homeomorphism to a Eucliean
space). Specifically, if $S$ is an oriented surface of genus $g$ with $b$ boundary components
and $n$ punctures in the interior, then \cite{fm08}.
$\mathcal{T}_S = R^{6g-6+2n+3b}$.
And we know (\cite{fm08}) that the mapping class group $Mod(S)$ acts properly discontinuously
on Teichmuller space $\mathcal{T}_S$, and the stabilizer is finite for each point, so by the \textit{Borel construction}
$E := \mathcal{T}_S \times_{Mod(S)} EMod(S)$, $EMod(S)$ is a universal principle $Mod(S)$ bundle
and $BMod(S) = EMod(S)/Mod(S)$ is the \textit{classifying space} of $Mod(S)$ (which classifies
all principle $Mod(S)$ bundles on any topological space). The $Mod(S)$ acts via diagonal action. Since $\mathcal{T}_S$ is contractible, $E$ is homotopic to $BMod(S)$, and the projection
$E \rightarrow \mathcal{T}_S/Mod(S) = \mathcal{M}_S$ is a fibration with fiber $EMod(S)/stab(x) = Bstab(x)$. Since
$stab(x)$ is a finite group, the cohomology of the fiber vanishes over $Q$. That is, we have
$H_*(Mod(S),Q) = H_*(BMod(S),Q) = H_*(MS,Q)$.

In the rest of the paper, we will consider the surfaces with possibly punctures/marked points
on the boundary.

For mapping class groups, we have the Birman exact sequence:
\begin{equation}
1 \rightarrow \pi_1(S^b_{g,n-1}) \rightarrow Mod(S^b_{g,n}) \rightarrow Mod(S^b_{g,n-1}) \rightarrow 1
\end{equation}
except the degenerate cases: $g = 0,b + n \leq 3 , g = 1,b + n \leq 1$.

Forgetting multiple punctures, one has
\begin{equation}
1 \rightarrow \pi_1(C(S^b_{g,n-1},k)) \rightarrow Mod(S^b_{g,n+k-1}) \rightarrow Mod(S^b_{g,n-1}) \rightarrow 1
\end{equation}
where for any surface $S$, $C(S,k)$ is the configuration space of $k$ distinct,ordered points
in $S$. That is,
\begin{equation*}
C(S,k) = S^{¡Ák} - BigDiag(S^{¡Ák})
\end{equation*}
where $S^{¡Ák}$
is the $k$-fold cartesian product of $S$ and BigDiag($S^{¡Ák}$) is the \textit{Big Diag} of
$S^{¡Ák}$
, that is, the subset of $S^{¡Ák}$ with at least two coordinates being equal. A variant of it is
\begin{equation}\label{2}
1 \rightarrow \pi_1(UTS^{b}_{g,n}) \rightarrow Mod(S^b_{g,n}) \rightarrow \mathcal{M}_{S^{b-1}_n} \rightarrow 1
\end{equation}
where $UTS^{b-1}_{g,n}$ is the unit tangent bundle (spherized tangent bundle) of $S^{b-1}_{g,n}$. The special cases above are also excluded.
Those sequences are important tools for our calculation of mapping class groups. We
will derive a sequence similar in spirit, for mapping class groups of surfaces with punctures/marked points on the boundary. More precisely, assume $S$ is an oriented surface
of genus $g$ with $b$ boundary components and $n$ punctures/marked points in the interior
and $m = (m_1,m_2,\dots, m_b)$ punctures/marked points on the boundary, with $m_i$ punctures/marked points on the $i$-th boundary component. $Mod(S)$ is the mapping class
group, i.e., the group of equivalence classes of oriented-preserving homeomorphisms of $S$ which fixes boundary component set-wise and fixes each punctures individually, the
equivalence relation is given by isotropy of the same type between them.
\newline
\newline
The following fact is obvious.
\begin{equation}\label{3}
Mod(S^{b,(m_1,m_2,\dots,m_i+1,\dots,m_b)}_{g,n}) \cong Mod(S^{b,(m_1,m_2,\dots,m_i,\dots,m_b)}_{g,n})  \quad \text{if} \quad m_i > 0
\end{equation}
So one can reduce the homology groups of moduli spaces of surfaces with $\vec{m} = (m_1,m_2,\dots,m_b)$
punctures/marked points to that with each $m_i$ at most 1. For this kind of surfaces we
have the following short exact sequence:
\begin{equation}\label{4}
(Z^1)^r \rightarrow Mod(S^{b,\vec{m})_{g,n}})\rightarrow Mod(S^b_{g,n}) \quad \text{if} \quad \chi(S^{b,n}_g) < 0
\end{equation}
$r$ is the number of $i$ for which $m_i = 0$, and $m_i = 0$ or 1 for $i = 1, 2,\dots, b$.

In order to achieve the desired homological dimension estimate, we need the following two key ingredients:
\begin{enumerate}[(i)]
\item The short exact sequence of groups corresponding to the fibration sequence of classifying spaces.
\newline
If $G_1,G_2,G_3$ are groups,
\begin{equation}
1 \rightarrow G_1 \rightarrow G_2 \rightarrow G_3 \rightarrow 1
\end{equation}
then the induced sequence
\begin{equation}
BG_1 \rightarrow BG_2 \rightarrow BG_3
\end{equation}
is a fibration, where $BG_i$ are classifying spaces of $G_i$.
\newline
This is the standard fact in classifying space theory.

\item If $F \rightarrow E\rightarrow B$ is a Serre fibration, $H_i(B,Q) = 0, H_j(F,Q) = 0$ when
$i \geq n,j \geq m + 1$, and $H_*(B,Q)$ is of finite rank, then $H_k(E,Q) = 0$ if $k \geq n + m$. In
fact, by the Serre spectral sequence,
\begin{equation}
E^2_{p,q} = H_p(F,H_q(B,Q)) \Rightarrow H_{p+q}(E,Q)
\end{equation}
the homology group we are going to compute is of index $p + q = k \geq n + m$, so either $p \geq n$
or $q \geq m + 1$. Using the assumption of homology groups $H_i(B,Q)$, $H_j(F,Q)$ and the rank finiteness of
$H_*(B,Q)$, we immediately get the conclusion.
\end{enumerate}

With all the preparation above, we can now prove the theorem \ref{hd_1}.
\begin{proof}
As before, $S^{b,\vec{m}}_{g,n}$ is a surface with possibly punctures/marked points on the interior and boundary.
\begin{enumerate}
\item
there are some $i$ such that $m_i \geq 2$. In this case, we can use \eqref{3} to reduce it to
the case that $m_i = 0$, or 1 for all $i$, i.e., $Mod(S^{b,\vec{m}}_{g,n}) \cong Mod(S^{b,\vec{m'}}_{g,n})$
, where $m¡ä_i = 1$, if $m_i \geq 1$;
\item
 all the $m_i = 0$ or 1. In this case, we use the short exact sequence \eqref{4}, provided
that the Euler class of punctures "filled in" surface is negative, and the induced fibration of
classifying spaces mentioned before. Since the classifying space of group $Z^r$ is $(S^1)^r
(i.e, K(Z^r, 1))$, we have a fibration
\begin{equation}
(S^1)^r \rightarrow \mathcal{M}_{S^{b,\vec{m}}_{g,n}}\rightarrow \mathcal{M}_{S^b_{g,n}}
\end{equation}
Since $(S^1)^r$ is an $r$ dimensional manifold, we have $H_i((S^1)^r,Q) = 0$ if $i \geq r + 1$. Using \eqref{4}, we have an induction:
\newline
if $H_i(\mathcal{M}_{S^b_{g,n}},Q) = 0$, for $i \geq 6g - 7 + 2n + 3b$, then $H_j(\mathcal{M}_{S^{b,\vec{m}}_{g,n}},Q) = 0
$ for $j \geq 6g - 7 + 2n + 3b + r$, where the condition $\chi(S^b_{g,n}) < 0$ holds.
\item
 there are no punctures on the boundary, and the surface has boundary. In this case, we
can use \eqref{2}, provided that the Euler characteristic of the disk ¡±patching¡± surface is negative,
and again, the induced fibration of classifying spaces. The classifying space of group
$\pi_1(UTS^{b-1}_{g,n})$ is $UTS^{b-1}_{g,n}$ (since $UTS^{b-1}_{g,n}$ is $K(\pi_1, 1)$), so we have a fibration sequence
\begin{equation}
UTS^{b-1}_{g,n} \rightarrow \mathcal{M}_{S^b_{g,n}} \rightarrow \mathcal{M}_{S^{b-1}_{g,n}}
\end{equation}
Because $UTS^{b-1}_{g,n}$
is 3 dimensional manifold, $H_i(UTS^{b-1}_{g,n},Q) = 0$
, for $i \geq 4$. Using \eqref{2}, we
also get an induction:
\newline
if $H_i(\mathcal{M}_{S^{b-1}_{g,n}},Q) = 0$, for $i \geq 6g - 7 + 2n + 3(b - 1)$, then $H_j(\mathcal{M}_{S^b_{g,n}}
,Q) = 0$, for $j \geq 6g - 7 + 2n + 3b$, where the condition $\chi(S^{b-1}_{g,n}) < 0$ holds.
Finally, $b = 0$. In this case, in \cite{cos07}, it has established that $H_i(\mathcal{M}_{S_{g,n}},Q) = 0$,
for $i \geq 6g - 7 + 2n$ except $(g,n) = (0, 3)$.
\end{enumerate}
\vspace{6pt}
Now assume our $S^{b,\vec{m}}_{g,n}$ has $\chi(S^{b,\vec{m}}_{g,n}) < 0$.
\newline
\newline
(i) There are no punctures on the boundary of $S^{b,\vec{m}}_{g,n}$. Then using the induction (as above)
we conclude that $H_i(\mathcal{M}_{S^{b,\vec{m}}_{g,n}},Q) = 0$ for $i \geq 6g - 7 + 2n + 3b$, except $(g,n,b) =
(1, 1, 0),(1, 0, 1),(1, 0, 0),(0, 3, 0),(0, 2, 0),(0, 1, 0),(0, 0, 0),(0, 2, 1),(0, 1, 1),(0, 0, 1),\\
(0, 1, 2),(0, 0, 2),(0, 0, 3)$.
As the Euler characteristic should be less than zero, we have
$(g,n,b) = (1, 1, 0),(1, 0, 1),(0, 3, 0), (0, 2, 1),(0, 1, 2),(0, 0, 3)$.
\newline
\newline
(ii) For other cases, the mapping class group is the same as the one of surfaces with fewer punctures
on the boundary(see \eqref{3} for precise meaning of this), and then we can still reduce it to
the one of surfaces without punctures on the boundary at all if $\chi(S^b_{g,n}) = 2 - 2g - (b + n) < 0$,
this includes all the cases except $(g,n,b) = (0, 1, 1),(0, 0, 2),(0, 0, 1)$. Note that $b > 0$. So
if $(g,n,b)$ is not those numbers, according to (i), we can reduce the surface with punctures on the boundary
to the one without punctures on the boundary by (i), except those ¡±bad¡± cases. Other
surfaces all satisfy the homological dimension estimates. Then by induction, the original
surface satisfies the homological dimension estimate.
So we are left the cases $(g,n,b, \vec{m}) = (0, 1, 1,\vec{m}),(0, 0, 2,\vec{m}),(0, 0, 1,\vec{m}),(1, 0, 1,\vec{m}), (0, 2, 1,\vec{m}),(0, 1, 2,\vec{m})$,and $(0, 0, 3,\vec{m})$. We need to treat them individually.
\newline
\newline
Let¡¯s start from $g = 1$.
\newline
\newline
If $(g,n,b, \vec{m}) = (1, 0, 1,\vec{m})$, then since the dimension of the moduli space $\mathcal{M}_{S^1_{1,0}}$ is 3 (it is the complement of trileaf knot) and $H_2(\mathcal{M}_{S^1_{1,0}}, Q) = H_3(\mathcal{M}_{S^1_{1,0}},Q) = 0$ (see \cite{mk02}), so in this case the homological dimension estimate is satisfied.
\newline
\newline
For $g = 0$, if $(g,n,b, \vec{m}) = (0, 0, 3, \vec{m})$, then since the moduli space of $S^3_{0,0}$, i.e., a pair of
pants, is the same as the Teichmuller space of $S^3_{0,0}$, for which $\mathcal{T}_{S^3_{0,0}} = R^3$
, so the homological dimension estimate is obviously satisfied. Then for all $\vec{m}$, the homological dimension is
satisfied. The same is true for $(g,n,b, \vec{m}) = (0, 1, 2,\vec{m})$.
\newline
\newline
if $(g,n,b, \vec{m}) = (0, 2, 1,\vec{m})$, then since $\mathcal{M}_{S^1_{0,2}} = R$, $H_i(\mathcal{M}_{S^1_{0,2}},Q) = 0$ when $i \geq 1$,
so when $m \geq 2$, the homological dimension estimate is satisfied.
The only cases left is $(0, 2, 1, 0),(0, 2, 1, 1)$
\newline
\newline
if $(g,n,b, \vec{m}) = (0, 0, 1, \vec{m})$, then the condition $\chi(S^{1,\vec{m}}_{0,0}) < 0$ requires $m \geq 3$, and the dimension of the moduli space of $S^{1,3}_{0,0}$
is 0, so when $m > 4$, the homological dimension estimate is
satisfied. This leaves the cases $(g,n,b, \vec{m}) = (0, 0, 1, 3),(0, 0, 1, 4)$ (these cases do not satisfy the
homological dimension estimate by dimension counting).
\newline
\newline
if $(g,n,b, \vec{m}) = (0, 0, 2,\vec{m})$, then the moduli space of $S^2_{1,0}$
is an interval $(0, 1)$ (\cite{liu02}),
of which the homology group $H_i(\mathcal{M}_{S^{2,{1,0}}_{0,0}},Q) = 0$ when $i \geq 1$. And the moduli space of $S^{2,(1,1)}_{0,0}$
is of
dimension 2 and homotopic to $S^1$
, so the homology group $H_i(\mathcal{M}_{S^{2,(1,1)}_{0,0}},Q) = 0$ when $i \geq 2$.
Thus when $m \neq (0, 1),(1, 0),(1, 1)$, the homological dimension estimate is satisfied. This
only leaves $(g,n,b, \vec{m}) = (0, 0, 2,(1, 0)),(0, 0, 2,(1, 1))$ (which do not satisfy the homological
dimension estimate).
\newline
\newline
if $(g,n,b, \vec{m}) = (0, 1, 1,\vec{m})$, because $m \geq 1$ and the dimension of the moduli space
$\mathcal{M}_{S^{1,1}_{0,1}}$
is 0, we have that when $m \geq 3$, then the homological dimension estimate is satisfied.
This leaves $(g,n,b, \vec{m}) = (0, 1, 1, 1),(0, 1, 1, 2)$.(these two cases do not satisfy the homological dimension estimate).
\newline
\newline
If there are no punctures on the boundary, then when $(g,n,b) = (1, 1, 0)$,\\$(1, 0, 1),(0, 1, 2),
(0, 0, 3)$, it is shown in \cite{cos07} and the result above that the homological dimension estimate is satisfied. When $(g,n,b) = (0, 3, 0),(0, 2, 1)$, the homological dimension estimate
is not satisfied by dimension counting.
\newline
\newline
In summary, $H_i(\mathcal{M}_{S^{b,\vec{m}}_{g,n}},Q) = 0$, for $i \geq 6g - 7 + 2n + 3b + m$, except $(g,n,b, \vec{m}) =
(0, 3, 0, 0),(0, 2, 1, 0),(0, 2, 1, 1), (0, 0, 2,(1, 0)),(0, 0, 2,(1, 1)),(0, 1, 1, 1),(0, 1, 1, 2),\\
(0, 0, 1, 3),(0, 0, 1, 4)$.
\begin{remark}
In fact, it is already from the above analysis that the homological dimension estimates can be improved to $H_i(\mathcal{M}_{S^{b,\vec{m}}_{g,n}},Q) = 0$ when $i \geq 6g - 7 + 2n + 3b + q$, $q$ is
the number of nonzero entries in $\vec{m}$.

\end{remark}
\end{proof}

\section{Estimate II}
In this section, we will get another homological dimension estimate for moduli space over coefficient $Q$, which is stronger than the one we have proposed above.
\newline
\newline
Denote
	$$
	d(g,n,b) = \left\{
	\begin{aligned}
	4g-5,  &     & n=b=0 \\
	4g + 2b + n - 4, &     & g>0,  n+b>0 \\
	2b + n-3, &     & \text{otherwise}
	\end{aligned}
	 \right.
	$$
	 and $q$ is the number of nonzero entries in $\vec{m}$.
\begin{theorem}
	$H_i(\mathcal{M}_{S_{g,n}^{b,\vec{m}}}, Q) = 0$ when $i>d(g,n,b) +q$.
\end{theorem}
\begin{proof}
(1)  It is easy to see that the moduli space $\mathcal{M}_{S^{b,\vec{m}}_{g,n}}$
is homotopy equivalent to $\mathcal{M}_{S^{b,\vec{m}}_{g,n}}$
, where $m¡ä_i = 0$ or 1, depending on whether $m_i = 0$ or $m_i \ge 0$. And $\mathcal{M}_{S^{b,\vec{m}}_{g,n}}$
is a fibration with fiber homeomorphic to $(S^1)^p$ over moduli space $\mathcal{M}_{S^b_{g,n}}$
, where $p$ is the number of non-zero entries in $\vec{m¡ä}$
. (2) In \cite{har86}, it is shown that when $\chi(S^b_{g,n}) < 0$
then the $Mod(S^b_{g,n})$ is a virtual duality group of dimension $d(g,n,b)$, in particular,
the virtual cohomological dimension of $Mod(S^b_{g,n})$ is $d(g,n,b)$. Since for a duality
group the cohomological dimension and homological dimension agrees (\cite{be73}), we have
$H_i(\mathcal{M}_{S^b_{g,n}},Q) = 0$, if $i > d(g,n,b)$. Summarizing the above two, and using the Serre spectral
sequence and (ii) in the proof of theorem \eqref{hd_1}, we get the conclusion. We still have some cases which need to be
treated separately, they are $(g,n,b, \vec{m}) = (0, 1, 1,\vec{m}),(0, 0, 2,\vec{m}),(0, 0, 1,\vec{m})$.
\newline
\newline
When $(g,n,b, \vec{m}) = (0, 1, 1, \vec{m})$ we need only consider $m > 0$. In this case, since
$H_*(\mathcal{M}_{S^{b,\vec{m}}_{g,n}},Q) = H_*(\mathcal{M}_{S^{b,(1)}_{g,n}},Q)$ and $S^{b,(1)}_{g,n}$ is of dimension 0, so $H_i(\mathcal{M}_{S^{1,\vec{m}}_{0,1}},Q) = 0$ when
$i \geq 1$. Since $d(0, 1, 1) = 0,d(0, 1, 1) + q = 1$, the result is stronger.
\newline
\newline
When $(g,n,b, \vec{m}) = (0, 0, 2,\vec{m})$, $\vec{m}$ must has at least one nonzero entry. From the result
in the proof of theorem 1, we have $H_i(\mathcal{M}_{S^{2,(m,0)}_{0,0}},Q) = 0$ when $i \geq 1$, it is stronger. And
we have $H_i(\mathcal{M}_{S^{2,(m1,m2)}_{0,0}},Q) = 0$ (i.e.,$q = 2$) when $i \geq 2$, which is also stronger.
\newline
\newline
When $(g,n,b, \vec{m}) = (0, 0, 1, \vec{m})$, then we have $m \geq 3$. When $m = 3, S^{1,\vec{m}}_{0,0}$
is of dimension 0, so $H_i(\mathcal{M}_{S^{1,\vec{m}}_{0,0}},Q) = 0$ when $i \geq 1$, which is the same as the above estimate.

\end{proof}

In this paper, we further discuss moduli spaces of surfaces with \textit{unordered} punctures and
boundary components, which is quite useful when the boundary and marked points are equipped with group actions of permutations. We similarly let $\tilde{S}^{b,\vec{m}}_{g,n}$ be a surface with unordered punctures
and boundary component.
\begin{theorem}
$H_i(\mathcal{M}_{\tilde{S}^{b,\vec{m}}_{g,n}},Q) = 0$ when $i > d(g,n,b) + q$
\end{theorem}
This can be shown as follows. 
\begin{proof}
It is an exact sequence
\begin{equation}
0 \rightarrow Mod(S^{b,\vec{m}}_{g,n}) \rightarrow Mod(\tilde{S}^{b,\vec{m}}_{g,n}) \rightarrow S_n \times S_b \ltimes (Z_{m_1} \times \cdots \times Z_{m_b}) \rightarrow 0
\end{equation}
The map $Mod(S^{b,\vec{m}}_{g,n}) \rightarrow Mod(\tilde{S}^{b,\vec{m}}_{g,n})$
is the canonical one, and $Mod(\tilde{S}^{b,\vec{m}}_{g,n}) \rightarrow S_n \times S_b \ltimes (Z_{m_1} \times \cdots \times Z_{m_b})$ records how the homeomorphism permutes punctures.

The third nonzero term of the exact sequence being a finite group implies that if the homology group
over $Q$ of the second nonzero term is 0 then so it is for the first nonzero term. This can
be deduced from the ¡±transfer map¡± for group homology associated to a group G and
its finite-index subgroup H (see \cite{brown82},\cite{aj04}). The composition of the restriction map with the
transfer map is: $Tr^G_H \circ Res^G_H(x) = [G : H]x$, for $x \in H_*(G,Q)$. Because $H_*(G,Q)$ is a $Q$
vector space, it means that the restriction map, $Res^G_H : H_*(G,Q) \rightarrow H_*(H,Q)$, is injective.
\end{proof}

\end{document}